\documentclass{amsart}
\usepackage{amsmath,amssymb}

\usepackage[latin1]{inputenc}

\newcommand{\N}{{\mathbb N}}

\newcommand{\R}{{\mathbb R}}
\newcommand{\C}{{\mathbb C}}
\newcommand{\K}{{\mathbb K}}
\newcommand{\OK}{{\overline{{\mathbb K}}}}
\newcommand{\OR}{{\overline{{\mathbb R}}}}

\newcommand{\Ot}{{\widetilde{\Omega}}_c}

\newcommand{\I}{{\bf{I}}}

\newcommand{\vv}{\vspace{0.6cm}}

\newcommand{\Hq}{{\mathbb H}}
\newcommand{\OHq}{{{\overline{\mathbb H}}}}

\newtheorem{theorem}{Theorem}[section]
\newtheorem{lemma}[theorem]{Lemma}
\newtheorem{proposition}[theorem]{Proposition}
\newtheorem{corollary}[theorem]{Corollary}
\newtheorem{definition}[theorem]{Definition}
\newtheorem{remark}[theorem]{Remark}
\newtheorem{example}[theorem]{Example}

\begin{document}

\title[The Colombeau  Quaternion Algebra]{The Colombeau  Quaternion  Algebra}

\author{Cortes W.}
\address{Department of Mathematics,
 Universidade Estadual de Maring\'a,  87020-900, Maring\'a, PR, Brazil} \email{wocorollarytes@uem.br}
 \author{Ferrero M.A.}
\address{ Instituto de Matem\'atica e Estat\'\i stica,
 UFRGS,
Caixa Postal 66281, S\~ao Paulo, CEP 05315-970 -
Brazil} \email{ferrero@mat.ufrgs.br}
\thanks{Research partially
supported by
CNPq-Brazil.
 \author{Juriaans S.O.}
\address{ Instituto de Matem\'atica e Estat\'\i stica,
 Universidade de S\~ao Paulo,
Caixa Postal 66281, S\~ao Paulo, CEP 05315-970 -
Brazil} \email{ostanley@ime.usp.br}
} \subjclass{Primary 46F30;
Secondary 46F20. Keywords and phrases: Colombeau algebra, generalized function, sharp topologies, quaternions, dense ideal, holomorphic function.}

\begin{abstract}We introduce and investigate the  topological  algebra of Colombeau Generalized quaternions,  $\OHq$. This is an important object to study if one wants to build the algebraic theory of Colombeau generalized numbers. We classify the dense ideals of $\OK$, in the algebraic sense and prove that it has a maximal ring of quotients which is Von Neumann regular. Using the classification of the dense ideals  we give  a criteria for a generalized holomorphic function to satisfy the identity theorem.
\end{abstract}

\maketitle

\section{{\bf{Introduction}}}

Since its introduction, the theory of Colombeau generalized functions has undergone rapid grow. Fundamental for the
theory were the definition of  Scarpalezos' sharp topologies  and the notion of point value
by Kunzinger-Obberguggenberger. A global theory was   developed in \cite{gksv}.

The study of the algebraic aspects of this theory however are relatively recent. This was proposed
independently by both J. Aragona and M. Obberguggenberger and   started with a paper by Aragona-Juriaans
(see \cite{AS}). This, and the developments mentioned above (see \cite{ko,S}), due to 
D. Scarpalezos, M.  Kunzinger and  M. Obberguggenberger,  led  Aragona-Fernandez-Juriaans to propose a differential
calculus which in its turn was used to continue the algebraic aspects of the theory (see \cite{afj,ASS}).

In this paper we focus on an algebra that may play an important role  in the study of the algebraic theory of these
algebras. We introduce the Colombeau generalized quaternion algebras, $\OHq$, and study its topological and algebraic
properties. We  further classify the dense ideals, in the algebraic sense, of $\OK$ and prove that $\OK$ and $\OHq$ have
a martindale ring of quotients. Using the classification of the dense ideals, we give a  criteria for a generalized
holomorphic function to satisfy the identity theorem.

The notation used is mostly standard. Some important references for
the theory of Colombeau generalized numbers, functions and their
topologies are \cite{AB}, \cite{JC2}, \cite{gksv},  \cite{ku},
\cite{mog} and \cite{S}.

\section{{\bf{Algebraic  Theory}}}

In this section we recall  some algebraic theory of Colombeau generalized functions. We refer the interested reader to  \cite{afj}, \cite{ASS} and \cite{AS} for notation, more details and proofs of the results mentioned here.

The norm of an element
$x\in \OK$ is defined by $||x||:=D(x,0)$, where $D$ is the ultra metric in $\OK$ defined by Scarpalezos. Denote by $Inv(\OK )$ the unit group of $\OK$.

 \typeout{$\OK\ $, equipped with the ultra metric defined above, is a   complete commutative  topological $\K$-algebra. It is neither  artinian nor  noetherian, it is not a domain and its Jacobson and prime radical are trivial (\cite{AS}).}

\begin{theorem}[Fundamental Theorem of $\ {\overline{\K}}$] {\rm [   \cite{AS} ]}

Let $x\in \OK$ be any element. Then one of the following holds:\\
{\bf{1.}} $x\in Inv(\OK )$.\\
{\bf{2.}} There exists an idempotent $e\in \OK$ such that $\ x\cdot e=0$.

\noindent Moreover $Inv(\OK )$ is an open and dense in $\ \OK$.
\end{theorem}

\begin{theorem}{\rm \cite{AS}}  An element $x\in \OK$ is a unit if and only is there exists $a\geq 0$ such that $|\hat{x}(\varepsilon )|\geq \varepsilon^a$, for $\varepsilon$ small enough.
\end{theorem}

Let ${{\mathcal{S}}}:= \{S\subset \I |\  0 \in \overline{S}\bigcap \overline{S^c}\}$ where the bar denote topological closure. We than denote by $P_*(\mathcal{S})$ the set of all subsets $\mathcal{F}$ of $\mathcal{S}$ which are stable under finite union and such that if $S\in \mathcal{S}$ then either $S$ or $S^c$ belongs to $\mathcal{F}$. By $g(\mathcal{F})$ we denote the ideal generated by the characteristic functions of elements of $\mathcal{F}$.

\begin{theorem}{\rm \cite{AS}}
Let $\mathcal{P}\lhd \OK$ be a prime ideal. Then \\
{\bf{1.}} There exists ${\mathcal{F}}_0 \in P_*(\mathcal{S})$ such that $g({\mathcal{F}}_0 )\subset {\mathcal{P}}$.\\
{\bf{2.}}   $\{ \ \overline{g(\mathcal{F})}\ |\ {\mathcal{F}}\in P_*(\mathcal{S})\}$ is the set of all maximal ideals of $\ \OK$.\\
\end{theorem}

In \cite{ASS} it is proved that $g({\mathcal{F}}_0 )$ is in fact a minimal prime ideal of $\ \OK$. In general $g({\mathcal{F}}_0 )$ is not closed and so $\OK$ is not Von Neumann regular.

If $I\lhd \OK $ is a maximal ideal then $\K$ is algebraically closed in $\OK /I$ and from this it follows that ${\mathcal{B}}(\OK )$, the set of idempotents of $\OK$, does not depend on $\K$, i.e., ${\mathcal{B}}({\overline{\C }} )={\mathcal{B}({\overline{\R }} )}$ (see [\cite{AS}, section IV]). Moreover, in \cite{ASS} it is proved   that ${\mathcal{B}}({\overline{\C }} )=\{\chi_A|\ A\in {\mathcal{S}}\}$, where $\chi_A$ denotes the characteristic function of the set $A$.

  For the sake of completeness we recall the order structure of $\overline{\R}$ defined in \cite{ASS}.

\begin{lemma}
\label{order}

For a given $x\in {\overline{\R}}$ the following are equivalent:\\
{\bf{(i)}}  Every  representative ${\hat{x}}\ $of $x$ satisfies the condition\\
{\bf{(*)}} $\forall b>0, \exists \eta_b\in \I$ such that ${\hat{x}}(\varepsilon )\geq -\varepsilon^b$, whenever $0<\varepsilon <\eta_b$.\\
{\bf{(ii)}}   There exists a representative ${\hat{x}}$ of $x$ satisfying {\bf{(*)}}.\\
{\bf{(iii)}} There exists a representative $x_*$ of $x$ such that $x_*(\varepsilon )\geq 0, \forall \varepsilon \in \I$.
\end{lemma}

\begin{definition}
\label{porder}

An element $x\in \overline{\R }$ is said to be non-negative,  quasi-positive or q-positive, if it has a representative satisfying one of the conditions of Lemma~\ref{order}. We shall denote this by $x\geq 0$. We shall say also that $x$ is non-positive or q-negative if $-x$ is q-positive. If $\ y\in \overline{\R }$ is another element then we write $x\geq y$ if $\ x-y$ is q-positive and $x\leq y$ if $\ y-x$ is q-positive.

\end{definition}

\begin{proposition}[Convexity of ideals]
\label{conv}

Let $J$ be an ideal of $\ \OK$ and $x,y \in \OK$. Then\\
{\bf{1.}} $x\in J$ if and only if $\ |x|\in J$.\\
{\bf{2.}} If $\ x\in J$ and $|y|\leq |x|$ then $y\in J$.\\
{\bf{3.}} If $\ \K =\R$, $x\in J$ and $0\leq y\leq x$ then $y\in J$.

\end{proposition}

Let $r\in  \R$. Then  $\alpha_r\in \OR$ is the element having $\varepsilon \rightarrow \varepsilon^r$ as a representative.  It has the property that $||\alpha_r||=e^{-r}$ and $||\alpha_r x||=||\alpha_r||\cdot ||x||$, for any $x\in \OK$.  With this we have  that an element $x\in \OK$ is a unit if and only if there exists $r\in \R$ such that $x\geq \alpha_r$.

In \cite{ASS} it is also proved that for $0<x\in \OK$ there exists $y\in \OK$ such that $x=y^2$. In the next section we will   use freely some of the results in this section.

\section{{\bf The   Topological   Algebra of Colombeau Generalized Quaternions}}

Here $(\Hq , |  \cdot  |)$ will denote the classical ring of real quaternion with its usual metric and basis $\{1,i,j,k\}$ and, unless stated otherwise, $\K=\R$. Moreover, if ${\bf{R}}$ is a ring then $\OHq (${\bf{R}}$)$ denotes the quaternion algebra over ${\bf{R}}$

\begin{definition} A function $\hat{x}:\I\rightarrow \Hq$ is moderate  if there exists $a\in
\R$ such that $\lim\limits_{\varepsilon\rightarrow 0}
\frac{|\hat{x}(\varepsilon)|}{\varepsilon^{a}}=0$.
\end{definition}

\noindent Let $\mathcal{E}_{M}(\Hq)=\{\hat{x}:\I\rightarrow \Hq\ |\ \hat{x}\, \,
is\ moderate\}$ and $\mathcal{N}(\Hq)=\{ \hat{x}\in \mathcal{E}_{M}(\Hq)\  |\  \lim\limits_{\varepsilon\rightarrow 0}
\frac{|\hat{x}(\varepsilon)|}{\varepsilon^{a}}=0, \forall a\in \R \}$. It is easily seen that  if
$\hat{x}(\varepsilon)=\hat{x_{0}}(\varepsilon)+\hat{x_{1}}(\varepsilon)i+\hat{x_{2}}(\varepsilon)j
+\hat{x_{3}}(\varepsilon)k \in \mathcal{E}_{M}(\Hq) $ then  $\lim\limits_{\varepsilon\rightarrow
0}\frac{|\hat{x}(\varepsilon)|}{\varepsilon^{a}}=0$ if and only if
$\lim\limits_{\varepsilon\rightarrow
0}\frac{|\hat{x}_{i}(\varepsilon)|}{\varepsilon^{a}}=0, \forall i$ if and
only if $\hat{x}_{i}\in \mathcal{E}_{M}(\R)$, for all $i\in
\{0,1,2,3\}$.  Hence it is clear that  $\mathcal{E}_{M}(\Hq)\cong  \Hq(\mathcal{E}_{M}(\R))$.

In the same way we obtain that  $\mathcal{N}(\Hq)\cong \Hq(\mathcal{N}(\R))$. Therefore we have that
$ \mathcal{E}_{M}(\Hq)/\mathcal{N}_{M}(\Hq)\cong
\Hq(\mathcal{E}_{M}(\R)/\mathcal{N}_{M}(\R))=\Hq({\overline{\R}})$. We denote $\Hq(\overline{\R})$ by $\overline{\Hq}$. Let $x=x_{0}+x_{1}i+x_{2}j+x_{3}k\in \overline{\Hq}$. Its conjugate $\bar{x}:=x_{0}-(x_{1}i+x_{2}j+x_{3}k)$ and its   norm  $n(x):=\sqrt{x_{0}^{2}+x_{1}^{2}+x_{2}^{2}+x_{3}^{2}}=\sqrt{x\bar{x}}$.

\begin{lemma}
\label{norm}
 Let $x\in \overline{\Hq}$. Then
 \begin{enumerate}
 \item $x\in Inv(\overline{\Hq})$ if and only if
$\ n(x)\in Inv(\overline{\R})$. In this case,
$x^{-1}=(n(x))^{-1}\bar{x}$.

\item An element $x\in \overline{\Hq}$  is a zero divisor if and only
if n(x) is a zero divisor.
\end{enumerate}
In particular we have that an element of $\ \OHq$ is either a unit or a zero divisor.
\end{lemma}

\begin{proof} Only the second  statement requires a proof. Suppose that $n(x)$ is a zero divisor. Hence there  exists an idempotent $e\in \overline{\R}$   such that $n(x)e=0$. From this it follows that
$0=n(x)^2e=\bar{x}xe=(xe)\bar{x}=x(e\bar{x})=0$. If $e\bar{x}\neq 0$, then $xe\bar{x}=0$ and it follows that $x$ is a
zero divisor. On the other hand if $\bar{x}e=0$, then $xe=0$. So, $x$ is a zero
divisor.\end{proof}

It follows from the lemma above that if $x=x_{0}+x_{1}i+x_{2}j+x_{3}k\in \overline{\Hq}$ and one of the $x_i$'s is a unit then $x$ is a unit.

For an element $x\in \overline{\Hq}$,  let $A(x):=\{a\in
\mathbb{R}|\  \lim\limits_{\varepsilon\rightarrow
0}\frac{|\hat{x}(\varepsilon)|}{\varepsilon^{a}}=0\}$ and
$V(x):=sup(A(x))$. It is easily seen that either $A(x)=\R$ or there exists $a\in \R$ such that either $A(x) = ]-\infty ,a[$ or $ A(x)= \subseteq \ ]-\infty , a]$. Define $||x||=e^{-V(x)}$ and $d(x,y):=||x-y||$. It is easily seen that all definitions above make sense and that $\ d\ $ defines an ultra metric on $\overline{\Hq}$. Denote by  $d_\pi$ the product
metric on $\overline{\Hq}$ induced by the topology of $\ \OK$.

\begin{theorem}
\label{homeo}
$(\overline{\Hq},d)$ and $(\overline{\Hq}, d_\pi)$ are
homeomorphic as topological spaces, i.e. $d$ and $d_\pi$ define the same topology on $\overline{\Hq}$.\end{theorem}

\begin{proof} Let $x=x_{0}+x_{1}i+x_{2}j+x_{3}k\in (\overline{\Hq},d)$ and
$r>0$ and let $B_{r}(x)=\{y\in \Hq \ |\ d(x-y)<r\}$. If $y=y_{0}+y_{1}i+y_{2}j+y_{3}k\in B_r(x)$ then
  $V(x-y)>ln(1/r)$. Hence $ln(1/r)\in A(x-y)$ and so if
$a<V(x-y)$ then $ \lim\limits_{\varepsilon\rightarrow
0 }\frac{|\hat{x}-\hat{y}|}{\varepsilon^{a}}=0$; but this occurs  if and only if
$ \lim\limits_{\varepsilon\rightarrow
0}\frac{|\hat{x}_{i}-\hat{y}_{i}|}{\varepsilon^{a}}=0 \forall
i\in \{0,1,2,3\}$. So, $V(x_{i}-y_{i})\geq a>ln(1/r)$ $\forall i\in \{0,1,2,3\}$ and  therefore $d_{\pi}(x,y)<r$. Since the arguments above are easily reversible, the proof is
complete.\end{proof}

\begin{corollary}     $(\overline{\Hq},d)$ is a complete metric algebra.\end{corollary}

\begin{proof}This follows by
Theorem~\ref{homeo} and a result of D. Scarpalezos which assures that $\OK$ is complete.
\end{proof}

The following result shows that the unit group of $\OHq$ is very big.

\begin{theorem}
\label{unit}
 The unit group $Inv(\OHq )$ is an open and dense subset of $\ \overline{\Hq}$.\end{theorem}

\begin{proof} Let $x\in Inv(\OHq )$, where $x=x_{0}+x_{1}i+x_{2}j+x_{3}k$. Then $n(x)=\sqrt{x\overline{x}}\in
Inv(\overline{\R})$. Thus, by the Fundamental Theorem of $\OK$,
there exists $r>0$ such that $B_{r}(n(x))\subseteq
Inv(\overline{\R})$. Suppose that $Inv(\overline{\Hq})$ is not an
open set. Hence, still by the same theorem, we have that  $\forall
n\in \mathbb{N}^{*}$, there exists $x_{n}\in B_{1/n}(x)$ such that
$n(x_{n})\notin Inv(\overline{\R})$. Hence, $x_{n}\rightarrow x$,
and so, by theorem~~\ref{homeo} $x_{ni}\rightarrow x_{i}$. It
follows that $n(x_{n})=\sum_{i=0}^{3}\sqrt{(x_{ni})^{2}}\rightarrow
\sum_{i=0}^{3}\sqrt{x_{i}^{2}}=n(x)$,   a contradiction.

We now prove the density.  Suppose that there exists $r\in \R$ and $z\in \OHq$ such that
$B_r(z)\bigcap Inv(\OHq )=\emptyset$. If $x=n(z)$ then, since the norm is obviously a continuous function, we have
that there exists an open ball $B_s(x)\subset \OK$ such that $n^{-1}(B_s(y))\subset B_r(z)$. But this, by
Lemma~\ref{norm},  contradicts the density of $Inv(\OK )$ in $\OK$
\end{proof}

\begin{corollary}
\label{maxid}
 Let $M$ be an ideal of $\ \overline{\Hq}$. If $M$ is a maximal
ideal of $\ \OHq$, then $M=\overline{M}$.\end{corollary}

\begin{lemma}
 \label{reduced}

 Let $J$ be an ideal of $\ \overline{\Hq}$. If $n(x)\in J$, then
$x\in J$. The converse is true if $\ \overline{\Hq}/J$ is  reduced.
\end{lemma}

\begin{proof} Let $x=x_{0}+x_{1}i+x_{2}j+x_{3}k\in \overline{\Hq}$ with $x_{i}\in
\overline{\R}$ such that $n(x)\in J$. Since, $J\ni
n(x)=\sqrt{x_{0}^{2}+x_{1}^{2}+x_{2}^{2}+x_{3}^{2}}\geq
\sqrt{x_{i}^{2}}=|x_{i}|$, for all $i\in \{0,1,2,3\}$, then $J\cap
\overline{\R}\ni n(x)\geq |x_{i}|$, for all $i\in \{0,1,2,3\}$. So,
$x_{i}\in J\cap \overline{\R}$, because $J\cap \overline{\R}$ is
convex.

Conversely, if $x\in J$, then $n(x)^{2}=n(x)n(x)=x\overline{x}\in
J$. Since $J$ is reduced, then $n(x)\in J$. \end{proof}


\section{{\bf The Algebraic Structure of $\ \OHq$}}

In this section too, unless stated otherwise, $\K =\R$.
We start   proving that the boolean algebra of $\ \OHq$ equals that of $\ \OK$.

\begin{theorem}
\label{idemp}
 $\mathcal{B}(\overline{\Hq})=\mathcal{B}(\overline{\mathbb{R}})$.
  \end{theorem}

\begin{proof} We clearly have that $\mathcal{B}({\overline{\mathbb{R}}}) \subseteq\mathcal{B}(\overline{\Hq})
$. On the other hand, let
$e\in \mathcal{B}(\overline{\Hq})$. Then $n(1-e), n(e)\in \OK$ are   idempotents and hence there exist  $A,B\in {\mathcal{S}}$ such that $n(e)=\chi_A$ and $n(1-e)=\chi_B$. From this we have that
$0=n(e)\chi_{A^c}=n(e\chi_{A^c})=0$. So
$e\chi_{A^c}=0$ and thus $e=e\chi_{A}$.

On the other hand, since $\ e(1-e)=0$, we have that
 $\chi_{A}\chi_{B}=0$. Since
$e\chi_{A}=e$ and $(1-e)\chi_{B}=1-e$ we get that
$e\chi_{A}+(1-e)\chi_{B}=1$. Thus,
$\chi_{A}=e\chi_{A}+(1-e)\chi_{B}\chi_{A}=e\chi_{A}$ and hence
$e=e\chi_{A}=\chi_{A}$.\end{proof}

We shall now use the Fundamental Theorem of $\ \OK$ to give a complete description of the maximal ideals of $\OHq$.

\begin{theorem}
\label{maximal}
Let $M$ be a maximal ideal of $\ \OHq$. Then
$M=\Hq({\overline{g(\mathcal{F})}})$, for some ${\mathcal{F}} \in
\mathcal{P}_{*}(\mathcal{S})$.
\end{theorem}

\begin{proof} We clearly have that $M\cap \overline{\R}$ is a prime
ideal. Hence, there exists an unique $\mathcal{F}\in
\mathcal{P}_{*}(\mathcal{S})$ such that
$g(\mathcal{F})\subseteq M\cap \overline{\R}\subseteq M$.
So, $\ \overline{g({\mathcal{F})}}\subseteq
\overline{M}=M$. Thus,
$\overline{g(\mathcal{F})}=\overline{\R}\cap M$ and it
follows that $\overline{\R}\cap M$ is a  maximal ideal.

Since
$\Hq(\overline{g(\mathcal{F})})=\overline{g
(\mathcal{F})}+
\overline{g(\mathcal{F})}i+\overline{g(\mathcal{F})}j+\overline{g(\mathcal{F})}k$, we  have that
$\overline{\Hq}/\Hq (\overline{g(\mathcal{F})})\cong
\Hq(\overline{\R}/\overline{g(\mathcal{F})})$ which, since $\R$ is algebraicly closed in $\overline{\R}/\overline{g({\mathcal{F}})}$,  is a
simple ring and thus
$M=\Hq(\overline{g(\mathcal{F})})$.\end{proof}

Let $I\lhd \OHq$ be an ideal and denote by $n(I)$ the ideal of $\OK$ generated by the set $\{n(x)|\ x\in I\}$.

\begin{remark} Note that $\Hq(\overline{\R})/\Hq(g(\mathcal{F}))$ is
isomorphic to $\Hq(\overline{\R}/g(\mathcal{F}))$. We claim that
$\Hq(\overline{\R})/\Hq(g(\mathcal{F}))$ is completely prime. In
fact, let $a, b\in  \Hq(\overline{\R})/\Hq(g(\mathcal{F}))$  be such
that $ab=0$. Then $0=n(ab)=n(a)n(a)$. Hence $n(a)=0$ or $n(b)=0$ and
the claim follows.  Besides that, we have that $\Hq(g(\mathcal{F}))$
is a minimal prime ideal of $\ \overline{\Hq}$ because the prime
ideals $g(\mathcal{F})$ of $\ \overline{\R}$ are minimal prime
ideals of $\ \overline{\R}$.\end{remark}

Let ${\bf{R}}$ be a ring and denote by $U({\bf{R}})$ its  Brown McCoy radical, i.e.   $U({\bf{R}})$ is  the intersection of all ideals $M$ of ${\bf{R}}$ such that ${\bf{R}}/M$ is a
simple and unitary. Note that if ${\bf{R}}$ is commutative then $U({\bf{R}})$ coincide with  the Jacobson radical.

\begin{lemma} $U( \OHq )\cap \overline{\R}=U(\overline{\R})$. In particular  $U(\OHq)=(0)$.\end{lemma}

\begin{proof} Let $M$ be an maximal ideal of $ \OHq $. By
Proposition~\ref{maximal}  we have that $M\cap \overline{\R}$ is a
maximal ideal of $\ \overline{\R}$. Thus $U(\OHq )\cap
\overline{\R}\supseteq U(\overline{\R})$.

On the other hand, let $M$ be a maximal ideal of $\overline{\R}$.
We have that $\Hq(M)=\Hq(\overline{g(\mathcal{F})})$, for
some $\mathcal{F}\in P(\mathcal{S}_{*})$, is a maximal ideal of
$\overline{\Hq}$. Hence, $U(\overline{\Hq})\cap \overline{\R}\subseteq
U(\overline{\R})$. So, $U(\overline{\Hq})\cap
\overline{\R}=U(\overline{\R})$. By lemma~\ref{reduced} we have that  $ n(U(\OHq ))\subset U(\OK )$ and since, by \cite{AS},
$U(\overline{\R})=(0)$ we have that  $U(\overline{\Hq})=(0)$.\end{proof}

Recall that it was proved in \cite{ASS} that $\overline{\R}$ is not Von Neumann regular.

\begin{proposition}
\label{vonreg}
 The ring $ \OHq $ is not Von Neumann regular.
\end{proposition}

\begin{proof} Suppose that $\overline{\Hq}$ is a regular ring. Then for
each $0\neq a\in \overline{\R}$ there would exist
$y=y_{0}+y_{1}i+y_{2}j+y_{3}k\in \overline{\Hq}$ such that $a=a^2y$.
Thus, $a=n(a)=a^2(n(y))$, which  implies that $\overline{\R}$ is Von Neumann regular, a contradiction.\end{proof}

In the next   results we characterize the essential ideals of
$\ \overline{\R}$ and $\ \OHq $. We shall use the notation of \cite{beidar}.

\begin{lemma}

 Let $I$ be an ideal of $\ \overline{\R}$. Then $r_{\overline{\R}}(I)\neq
(0)$ if and only if there exists $\ e\in \mathcal{B}(\overline{\R})$
such that $I\subseteq \overline{\R}e$. Equivalently, $I$ is
essential if and only if $\ I$ is not contained in a principal idempotent
ideal. Moreover if $r_{\overline{\R}}(I)\neq
(0)$ then $\mathcal{B}(\overline{\R}) \cap r_{\overline{\R}}(I)\neq
(0)$
\end{lemma}

\begin{proof} If $\ I\subseteq \overline{\R}e$, for some $e\in \mathcal{B}(\overline{\R})$ then
$1-e\in r_{\overline{\R}}(I)\neq 0$.

Conversely, if $\ 0\neq x\in r_{\overline{\R}}(I)$ then $x$ must be a zero
divisor and it follows that there exists $A\in {\mathcal{S}}$ such that $x\chi_A=0$
or equivalently $x\chi_{A^c}=x$. We claim that $\chi_{A}\in
r_{\overline{\R}}(I)$. In fact, for any $y\in I$ we have that
$xy=0$. Thus, $x\chi_{A^c}y=0$. So, if we choose representatives, we have that
$\hat{x}(1-\chi_{A})\hat{y}\in \mathcal{N}(\R)$.

If $\chi_{A^c} \hat{y}\notin \mathcal{N}(\R)$ then there exists
 $a\in \mathbb{R}$ such that
$ \lim\limits_{\varepsilon \rightarrow
0}\frac{|\chi_{A^c} (\varepsilon)\hat{y}(\varepsilon)|}{\varepsilon^{a}}\neq
0$. However, $ \lim\limits_{\varepsilon\rightarrow
0}\frac{|\hat{x}(\varepsilon)||\chi_{A^c}(\varepsilon)\hat{y}(\varepsilon)|}{\varepsilon^{a+b}}= \lim\limits_{\varepsilon\rightarrow
0}\frac{|\hat{x}(\varepsilon)\chi_{A^c} (\varepsilon)|}{\varepsilon^{a}}
\frac{|\chi_{A^c}(\varepsilon)\hat{y}(\varepsilon)|}{\varepsilon^{b}}
= \lim\limits_{\varepsilon\rightarrow
0}\frac{|\hat{x}(\varepsilon) \chi_{A^c}(\varepsilon))}{\varepsilon^{b}}=0,\ \forall b \in \mathbb{R}$. Hence, $\hat{x} \chi_{A^c} \in
\mathcal{N}(\R)$, which implies $x \chi_{A^c} =0$. Thus,
$x=x\chi_{A}+x\chi_{A^{c}}=0$, a contradiction. So, $ \chi_{A^c} \hat{y}\in
\mathbb{N}$ and follows that $\chi_{A^c} y=0$, for all $y\in I$.
Thus, $ \chi_{A^c}\in r_{\overline{\R}}(I)$. Using this we have for $\ y\in
I$ that $y=y\chi_{A}+y\chi_{A^c} =y\chi_{A}\in
\overline{\R}\chi_{A}$.\end{proof}

 \begin{lemma}
 \label{fg}
 If $I$ is a proper finitely generated ideal of $\ \OK$ then it is contained in a principal idempotent ideal and hence is not essential. In particular essential ideals are not finitely generated.

 \end{lemma}

 \begin{proof} Let $I = \langle x_1,\cdots , x_n\rangle$ be a proper ideal. It follows that $x:= \sum |x_i|\in I$ and hence is a zero divisor. So there exists $e\in \mathcal{B}(\overline{\R})$ such that $xe=0$. But $|x_i|e\leq xe=0$ and hence $x_ie=0, \forall i$. It follows that $I\subset \OK (1-e)$.
 \end{proof}

\begin{lemma}

Let $I$ be an ideal of $\ \OHq $. Then $I\in
\mathcal{D}(\overline{\Hq})$ if and only if $\ n(I)\in
\mathcal{D}({\overline{\R}})$.
\end{lemma}

\begin{proof} Let $I\in \mathcal{D}(\overline{\Hq})$, $J=n(I)$ and let $K$ be a non-zero
ideal of ${\overline{\R}}$. Choosing $0\neq y\in K$ we have  that there exists $0\neq x\in
I\cap y\OHq$. Thus, $0\neq n(x)\in  J\cap
  K $, which implies that $J\in
\mathcal{D}(\overline{\R})$.

Conversely, if $I\notin \mathcal{D}(\overline{\Hq})$, then
$r_{\overline{\Hq}}(I)\neq (0)$. Thus, there exists $0\neq x\in
r_{\overline{\Hq}}(I)$ such that $xy=0$, for all $y\in I$ and
thus $0=n(x)n(y),\ \forall y\in I$. So, $0\neq n(x)\in
r_{\overline{\R}}(n(I))$ and hence  $n(I)\notin
\mathcal{D}(\overline{\R})$.\end{proof}

\begin{lemma}
\label{norm1}
Let $I\lhd \OHq$ be an ideal. Then $I\subset \OHq (n(I))$. Moreover if $\ I$ is semi-prime then $\ I= \OHq (n(I))$.

\end{lemma}

\begin{proof}
In fact, take an element
$x=x_{0}+x_{1}i+x_{2}j+x_{3}k\in I$. Hence, $n(I)\ni
n(x)=\sqrt{x_{0}^{2}+x_{1}^{2}+x_{2}^{2}+x_{3}^{2}}\geq |x_{i}|$,
for all $i\in \{0,1,2,3\}$. So, by convexity of ideal, $x_{i}\in n(I)$, for all $i\in
\{0,1,2,3\}$ and hence $\ x\in \Hq(n(I))$. In case $I$ is reduced then we apply Theorem~\ref{reduced} to get equality.
\end{proof}

\begin{proposition}
\label{essid}
Let $\ I\lhd \overline{\Hq}$ be an ideal. The following conditions are equivalent:
\begin{enumerate}
\item $I\notin
\mathcal{D}(\overline{\Hq})$.

\item $I\subseteq \overline{\Hq}e$, for some $\ e^2=e\in \overline{\R}$.
\end{enumerate}
\end{proposition}

\begin{proof} Suppose that $I\notin \mathcal{D}(\overline{\Hq})$. Then there
exists $e\in \overline{\R}$ such that $n(I)\subseteq
\overline{\R}e$. By lemma~\ref{norm1} we have that $I \subseteq\Hq(n(I))$.   Therefore,
$I\subseteq \Hq(n(I))\subseteq \OHq(\overline{\R}e)=\Hq(\overline{\R})e$.
The proof of the converse is trivial.
\end{proof}

Here we let $\K$ stand for $\R$ of $\C$. It follows from Proposition~\ref{essid} that the singular ideals $Z_r(\OK )=Z_r(\OHq )=0$  and hence ,by Theorem 2.1.15 of \cite{beidar}, we have that $Q_{max}(\OK )$ and $Q_{mr}(\OHq )$, the maximal  right rings of quotients of $\OK$, are Von Neumann regular. Note that since $\OK$ is commutative we have that $\OK$ is contained in its extended centroid.


\section{{\bf{Generalized Holomorphic Functions}}}

We start  recalling  the definition of holomorphic   and analytic functions. Here $\K$ shall always stand for
$\C$, $\Omega$ denotes a non-void open subset of
$\C$, ${\mathcal H}(\Omega )=\{f\in {\mathcal C}^1(\Omega;\,\C)\,:\,\overline{\partial}f=0\}$ and
${\mathcal HG}(\Omega)=\{f\in{\mathcal G}(\Omega;\, \C)\,:\,\overline{\partial}f=0\}$.

The following theorem classifies  completely the analytic functions. We refer the reader to \cite{afj} for exact
statement and more details.

\begin{theorem}[\cite{afj}]
\label{holom}
 Let $f\in{\mathcal HG}(\Omega)$. Then $ f $ is analytic if and only if $\ f$ is sub-linear.
\end{theorem}

In \cite{afj1} a Generalized Goursat Theorem is proved and so what lacks is an Identity Theorem.

\begin{definition}
 Let $\ f$ be a holomorphic function. We say that $\ f$  satisfies the Identity Theorem if $\ \forall z_0\in \Ot$ we
 have the pre-image $\ f^{-1}(f(z_0))$ has no accumulation  point.

We denote by $C_f(z_0):=\langle f^{(n)}(z_0)\ |\ n\in \N^*\rangle$ the ideal of $\ \OK$ generated by the coefficients
of the Taylor series of $\ f(z)-f(z_0)$. Here $f^{(n)}$ denotes the $n$-th derivative of $\ f$.
\end{definition}

\begin{example}
Let $A\in {\mathcal{S}}$ and define $f(z):=\chi_Az$. Then $\ f$ is analytic and
$f(\chi_{A^c}\alpha_n)=0, \ \forall n\in\N$. Since $\lim\limits_{\epsilon\rightarrow 0}\chi_{A^c}\alpha_n=0$ we have
that this $\ f$ does not satisfy the Identity Theorem.
\end{example}

\begin{theorem}
\label{holomor} Let $f\in{\mathcal HG}(\Omega)$. If $\ f$ satisfies
the Identity Theorem  then $\ \forall z_0\in \Ot$ we have that
$C_f(z_0)\in D(\OK )$.
\end{theorem}

\begin{proof} Suppose that $C_f(z_0)$ is not dense. Then there exists a non-trivial idempotent $\ e\in \OK $ such that
$ce=c, \ \forall c\in C_f(z_0)$. By theorem 5.7 of \cite{afj} we
have that $\ f (z)=\displaystyle\sum_{n\geq 0} \frac{
f^{(n)}(z_0)}{n!}(z-z_0)^n$. Define $x_n:=x_0+(1-e)\alpha_n$. Then
$(x_n)$ converges to $x_0$ and $f(x_n)=0,\ \forall n\in \N$. It
follows that $f$ does not satisfy the identity theorem.
\end{proof}

We now give some examples showing that the converse also holds.

 Let $p(z) = a_0 + a_1z$ with $a_1\in Inv(\OK )$. Then $C_f(z_0) = \langle a_1\rangle\in D(\OK
 )$ and $f(z)=f(z_0)$ has a single solution in $\OK$.

 Consider the quadratic polynomial $p(z)=a_0+a_1z+a_2z^2$ with $a_1$
 invertible and $ea_2=0$ for some non-trivial idempotent $e$. If we
 try to solve the equation $p(z) = p(z_0)$ then, using the hypothesis, we obtain that
 $e(z-z_0)=0$. So if $(z_n)_{n\geq 1}$ is a sequence converging to $z_0$ such
 that $p(z_n)=p(z_0)$ then, using the definition the derivative of
 $f$ at $z_0$, we get
 $p^{\prime}(z_0)=\lim\limits_{n\rightarrow \infty}\frac{p^{\prime}(z_0)(z_n-z_0)}{-\alpha_{\ln \|z_n-z_0\|}}$. Now
 multiplying with $e$ yields $a_1e=0$ and hence $e=0$, a
 contradiction.

 Suppose now that $p(z)=a_0+a_1z+a_2z^2$ with both $a_1$ and $a_2$
 invertible and $\|a_1^{-1}a_2\|<1/2$. Solving the equation
 $p(z)-p(z_0)$ we obtain $0=(z-z_0)[a_1+a_2(z-z_0)]=
 a_1(z-z_0)[1+a_1^{-1}a_2(z+z_0)]$. Since $\|z_0\| \leq1$ we have that
 $\|z+z_0\|\leq 2$ and so $[1+a_1^{-1}a_2(z+z_0)]$ is a unit. It
 follows that $z=z_0$ and thus we have just one solution.

 Yet in the case of a quadratic polynomial $p(z)=a_0+a_1z+a_2z^2$
 such that $\langle a_1,a_2\rangle = \OK$ it is easy to see that if
 there exists a point $z_0$ such that $p^{\prime}(z_0)=0$ then $a_2$
 is invertible.

 So we see that in some cases we have a converse of
 theorem~\ref{holomor}.


\section{{\bf{Annihilator Ideals in  Polynomial Rings in several
Variables}}}

During this section $R$ denote any ring with identity. We begin with
a result that can easily be checked.

\begin{lemma}
\label{anula}
\begin{enumerate}
\item Let $f,g \in R[x_{1},..,x_{n}]$. Then $fRg=0$ if and only
if \newline $fR[x_{1},...,x_{n}]g=0$.
\item Let $f,g \in R[x_{1},..,x_{i-1}]$. Then
$fR[x_{1},..,x_{i-1}]g=0$ if and only if \newline
$fR[x_{1},..,x_{i}]g=0$.
\end{enumerate}
\end{lemma}

\begin{theorem} Let $f\in R[x_{1},...,x_{n}]$. Then
$r_{R[x_{1},...,x_{n}]}(f(x)R[x_{1},...,x_{n}])\neq (0)$ \newline if and
only if $r_{R[x_{1},...,x_{n}]}(f(x)R[x_{1},...,x_{n}])\cap R\neq
(0)$.\end{theorem}

\begin{proof} Lemma~\ref{anula} will be used without explicit mentioning. It is sufficient to prove the theorem for
 $n=2$.

Let
$\ f=a_{0}(x_{1})+a_{1}(x_{1})x_{2}+....+a_{n}(x_{1})x_{2}^{n}$. If
degree of $f$ as a polynomial in $x_{2}$ is zero or $f(x)=0$, then
the assertion is clear. So, let $\deg_{x_{2}} (f)=m>0$. Assume, to
the contrary, that $r_{R[x_{1},x_{2}]}(f(x)R[x_{1},x_{2}])\cap
R=(0)$ and let $ g(x)=b_{0}(x_{1})+...+b_{m}(x_{1})x_{2}^{m}\in
r_{R[x_{1},x_{2}]}(fR[x_{1},x_{2}])$ be of minimal degree in
$r_{R[x_{1},x_{2}]}(f(x)R[x_{1},x_{2}])$\typeout{ as a polynomial in $x_{2}$,
such that $b_{m}(x_{1})\neq 0$}. Using   the same methods of
\cite{H}, we obtain that $fR[x_{1},x_{2}]b_{m}(x_{1})=0$,
a contradiction. So, $r_{R[x_{1},x_{2}]}(f(x)R[x_{1},x_{2}])\cap
R[x_{1}]\neq (0)$. Hence, there exists a polynomial $b(x_{1})\in
R[x_{1}]$ such that $fR[x_{1},x_{2}]b(x_{1})=0$. Thus,
$a_{i}(x_{1})R[x_{1}]b(x_{1})=0$, for all $i\in \{0,...,n\}$. Using the methods of
  \cite{H}, we have that $a_{i}(x_{1})R[x_{1}]b=0, \ \forall i\in \{0,...,n\}$, where $b$ is the leading coefficient of
$b(x_{1})$. Therefore, $fR[x_{1},x_{2}]b=0$.\end{proof} As an
application we have the following

\begin{lemma} Let $L$ be a linear operator over $\OK$ and $f$ a function whose
image has non-empty interior. If $L$ has non-trivial annilator, i.e., there exists a non-zero $a\in \OK$ such that $aL(u)=0, \forall u$, then
the equation $L(u)=f$ has no solutions.
\end{lemma} Note that this extends a result of \cite{afj}.

 \vv
 
\noindent{\underline{{\bf{aknowledgement:}}}} The first and the last
authors are grateful to IME-USP and UFRGS, respectively, were this work was
done.


\begin{thebibliography}{99}
\itemsep=-2pt



\bibitem {AB}
Aragona, J.,  Biagioni, H.,   {\it Intrinsic definition of the Colombeau algebra of generalized functions,} Anal. Math. 17, 2 (1991), 75-132.

\bibitem{afj} Aragona, J.,  Fernandez, R., Juriaans, S.O.,
  {\it A Discontinuous  Colombeau Differential Calculus},
   Monatsh. Math. 144, 13-29 (2005).


\bibitem{afj1} Aragona, J.,  Fernandez, R., Juriaans, S.O.,
  {\it Colombeau Differential Calculus and Applications}, preprint.

\bibitem{ASS} Aragona, J., Juriaans, S. O., Oliveira, O.R.B., Scarpelezos D.,
Algebraic and Geometric Theory of the topological algebra of Colombeau
generalized functions. submitted.


\bibitem{AS}Aragona J., Juriaans S. O., Some structural properties
of the topological ring of Colombeau generalized numbers. Comm.
in Algebra 2((2001), 2201-2230.

\bibitem{beidar} Beidar, K.I., Martindale III, W.S., Mikhalev, A.V., {\it Rings with generalized identities}, Pure and Applied Mathematics, 1995.


\bibitem{JC2} Colombeau, J.F., New Generalized Functions and Multiplication of Distributions, {\it North Holland}, Amsterdam 1984.



\bibitem{gksv} Grosser M, Kunzinger M., Steinbauer R., Vickers J.A.,
{\it  A global theory of algebras of generalized functions},
 Adv. Math., 166(1):50-72, 2002.


\bibitem{H} Hirano Y., On annihilators ideals of a polynomial ring over a
non commutative ring, J. of Pure and Applied Algebra 168 (2002),
45-52.



\bibitem {ku}
Kunzinger, M. ,
 {\it Lie transformation groups in Colombeau algebras,}{\it Lie transformation groups in Colombeau algebras,}
   Doctoral Thesis, University of Viena, 1996.toral Thesis, University of Viena, 1996.



\bibitem{ko} Kunzinger, M., Obberguggenberger, M.,



\bibitem {mog}
 Oberguggenberger, M.
 {\it Multiplication of distributions and applications to partial differential equations},  Pitman, 1992.

\bibitem{S}
Scarpalezos, D.,
 {\it Colombeau's generalized functions: topological structures micro local properties. A simplified point of view,}
  CNRS-URA212,
 Universit\'e Paris 7, 1993.

\end{thebibliography}
\end{document}